\documentclass[12pt,final]{article}
\usepackage{amsmath,amsthm,amsfonts,amssymb,graphicx,enumerate,psfrag}
\usepackage{mathrsfs}
\usepackage{fullpage}

\usepackage[colorlinks,citecolor=blue,urlcolor=blue]{hyperref}
\usepackage[utf8]{inputenc} 
\newtheorem{lemma}{Lemma}
\newtheorem{theorem}{Theorem}
\newtheorem{conjecture}{Conjecture}

\newtheorem{remark}{Remark}

\newcommand{\dR}{\mathbb {R}}
\newcommand{\tls}{{\rm t}_{\textsc{ls}}}
\newcommand{\trel}{{\rm t}_{\textsc{rel}}}

\newcommand{\tmls}{{\rm t}_{\textsc{mls}}}

\newcommand{\dist}{{\mathrm{dist}}}
\newcommand{\EE}{{\mathbf{E}}}

\newcommand{\E}{{\mathcal{E}}}
\newcommand{\X}{{\mathcal{X}}}
\newcommand{\LL}{{\mathcal{L}}}

\newcommand{\Ent}{{\mathrm{Ent}}}
\newcommand{\Lip}{{\mathrm{Lip}}}
\newcommand{\diam}{{\mathrm{diam}}}
\newcommand{\PP}{{\mathbf{P}}}
\newcommand{\cE}{\mathcal{E}}

\newcommand{\dX}{\mathbb {X}}
\title{Intrinsic regularity in the discrete log-Sobolev inequality}
\author{Justin Salez and Pierre Youssef}
\begin{document}
\maketitle
\begin{abstract}
The  \emph{chain rule} lies at the heart of the powerful Gamma calculus for Markov diffusions on manifolds, providing remarkable connections between several fundamental notions such as Bakry-\'Emery curvature, entropy decay, and hypercontractivity. For Markov chains on finite state spaces, approximate versions of this chain rule have recently been put forward, with an extra cost that depends on the log-Lipschitz regularity of the considered observable. Motivated by those findings, we here investigate the regularity of extremizers in the discrete log-Sobolev inequality. Specifically, we show that their log-Lipschitz constant is bounded by a universal multiple of $\log d$, where $d$ denotes the inverse of the smallest non-zero transition probability. As a consequence, we deduce that the log-Sobolev constant of any reversible Markov chain on a finite state space is at least a universal multiple of $\kappa/\log d$, where $\kappa$ is the Bakry-\'Emery curvature. This is a sharp discrete analogue of what is perhaps the most emblematic application of the Bakry-\'Emery theory for diffusions. We also obtain a very simple proof of the main result in \cite{MR4620718}, which asserts that the log-Sobolev constant and its modified version agree up to a $\log d$ factor. Our work consolidates the role of the sparsity parameter $\log d$ as a universal cost for transferring results from Markov diffusions to discrete chains.
\end{abstract}

\section{Introduction}

Functional inequalities play a central role in the analysis of Markov semi-groups and the concentration-of-measure phenomenon \cite{MR3185193,MR1767995,MR1849347}. In particular, the log-Sobolev inequality (LSI) and its modified version (MLSI)  \cite{MR2283379} are powerful tools for quantifying hypercontractivity and entropy decay respectively, with applications ranging from the analysis of mixing times of Markov chains to functional analysis and statistical physics. Let us recall their definition in the simple setting of finite state spaces. Throughout the paper, we consider an irreducible transition matrix $T$ on a finite state space $\dX$, satisfying the reversibility condition
\begin{eqnarray*}
\forall (x,y)\in\dX^2,\qquad \pi(x)T(x,y) & = & \pi(y)T(y,x),
\end{eqnarray*}
with respect to some probability measure $\pi$. We let $(P_t)_{t\geq 0}$ denote the corresponding continuous-time Markov semi-group, whose  generator acts on any $f\colon\dX\to\dR$ as follows:
\begin{eqnarray*}
(\LL f)(x) & := & \sum_{y\in\dX}T(x,y)\left(f(y)-f(x)\right). 
\end{eqnarray*}
The associated carr\'e du champ operator is then given by the formula
\begin{eqnarray*}
\Gamma(f,g)(x) & := &  \frac{1}{2}\sum_{y\in\dX}T(x,y)\left(f(y)-f(x)\right)\left(g(y)-g(x)\right),
\end{eqnarray*}
and integrating this expression against $\pi$ gives rise to the Dirichlet form  
\begin{eqnarray}
\label{Dirichlet}
\E(f,g) & = &   \EE\left[\Gamma(f,g)\right],
\end{eqnarray}
where $\EE[f]=\sum_{x\in\dX}\pi(x)f(x)$ denotes the expectation of $f$ with respect to $\pi$. 
As usual, we simply write $\Gamma(f)$ and $\cE(f)$ when $g=f$, and we define the entropy of $f\colon\dX\to[0,\infty)$ as
\begin{eqnarray*}
\Ent(f)& := & \EE[f\log f]-\EE[f]\log\EE[f].\end{eqnarray*}
Following the seminal works \cite{diaconis,MR2283379}, we define the inverse log-Sobolev constant $\tls$ and its modified version $\tmls$  as the smallest numbers such that the functional inequalities
\begin{eqnarray}
\label{LSI}
\Ent(f) & \leq & \tls\, \E(\sqrt{f}),\\
\label{MLSI}\Ent(f) & \leq & \tmls\, \E(f, \log f),
\end{eqnarray}
hold for all $f\colon\dX\to(0,\infty)$. In the completely different setting of Markov diffusions on Euclidean spaces or Riemannian manifolds, the generator $\LL$ is a second-order differential operator, and the associated carr\'e du champ operator classically satisfies the chain rule 
\begin{eqnarray}\label{eq: chain rule}
    \Gamma(f, \Phi(g)) & = & \Phi'(g) \Gamma(f,g),
\end{eqnarray}
for any smooth function $\Phi:\dR\to \dR$; see, e.g., the textbook \cite{MR3155209}. In particular, it follows that 
\begin{eqnarray}
\label{logchainrule}
\Gamma(f,\log f) & = & \frac{\Gamma(f)}{f} \ = \ 4\Gamma(\sqrt{f}),
\end{eqnarray}
 which readily implies that $\tls=4\tmls$. Let us emphasize that this remarkable equivalence is specific to diffusions: in general, we only have the one-sided inequality 
 \begin{eqnarray}
 \label{onesided}
 \tls & \ge & 4\tmls.
 \end{eqnarray}
Indeed, the log-Sobolev inequality (\ref{LSI}) quantifies the hypercontractive nature of the semi-group $(P_t)_{t\ge 0}$, whereas its modified version (\ref{MLSI}) controls the decay of entropy, a notion which is much weaker in the absence of the chain rule (see, e.g., \cite{MR2283379}). Nevertheless, it was recently observed in \cite{MR4620718} that an approximate version of (\ref{logchainrule}) remains valid in the discrete setting, albeit with an additional cost that depends on the spatial regularity of the considered observable. Specifically, one can easily check  that
\begin{eqnarray}
\label{approxlogchainrule}
\Gamma(f,\log f) & \le & \phi(\Lip(\log f))\,\Gamma(\sqrt{f}),
\end{eqnarray}
where $\Lip(g):=\max\left\{|g(x)-g(y)|\colon (x,y)\in\dX^2, T(x,y)>0\right\}$ denotes the Lipschitz constant of a function $g\colon \dX\to\dR$, and where we have introduced the (continuously increasing, near linear) cost function $\phi\colon[0,\infty)\to[4,\infty)$ defined by
\begin{eqnarray*}
\phi(r) & := & \frac{e^{\frac r 2}+1}{e^{\frac r2}-1}r.
\end{eqnarray*}
This observation motivates the investigation of the log-Lipschitz constant of extremizers in the log-Sobolev inequality (\ref{LSI}), and our first main contribution is a universal and sharp answer to this question. As in \cite{MR4780485,MR4620718,pedrotti2025newcutoffcriterionnonnegatively}, we introduce the sparsity parameter
\begin{eqnarray}
\label{def:d}
d & := & \max\left\{\frac{1}{T(x,y)}\colon (x,y)\in\dX^2,T(x,y)>0\right\},
\end{eqnarray}
which, in the special case where $T$ is the transition matrix of a simple random walk on a graph, corresponds to the maximum degree. Note that $d\ge 2$ unless $|\dX|=1$, in which case all results presented here become trivial. 
\begin{theorem}[Intrinsic regularity in the discrete LSI]\label{th: regularity of extremal}If $f\colon \dX\to(0,\infty)$ achieves equality in the log-Sobolev inequality (\ref{LSI}), then 
\begin{eqnarray*}
\Lip(\log f) & \le & 14\log d.
\end{eqnarray*}
\end{theorem}
As a first application, we combine this estimate with the approximate chain rule (\ref{approxlogchainrule}) to reverse the trivial inequality (\ref{onesided}), incurring a cost of order $\log d$. This provides a discrete counterpart to the celebrated equivalence $\tls=4\tmls$ that holds for Markov diffusions. 
\begin{theorem}[Quantitative equivalence between LSI and MLSI]\label{th:LSMLS}We always have
\begin{eqnarray*}
4\tmls \ \le \ \tls & \le & 15 \tmls \log d.
\end{eqnarray*}
\end{theorem}
This estimate  was the main result of the recent work \cite{MR4620718}, albeit with a slightly worst constant that ours. The proof in \cite{MR4620718} relied on a careful analysis of a regularization procedure first introduced in \cite{MR4765699}, whose effect was to improve the log-Lipschitz constant of any observable $f\colon\dX\to[0,\infty)$ to $O(\log d)$ while preserving the quantities $\Ent(f)$, $\cE(\sqrt{f})$ and $\cE(f,\log f)$ up to constant factors. Ironically, our main theorem reveals that this sophisticated procedure was in fact unnecessary, the relevant observables being already \emph{intrisically} log-Lipschitz regular with constant $O(\log d)$. As a consequence, we obtain a significantly shorter proof than the one in \cite{MR4620718}, while offering additional insight into the reason behind the appearance of the $\log d$ factor.

A less immediate, but genuinely new consequence of our main result is a sharp universal relation between the log-Sobolev constant of a discrete Markov chain and a fundamental geometric parameter known as the \emph{Bakry-\'Emery curvature}. The latter was defined in \cite{MR889476} as the largest number  $\kappa$ such that the sub-commutation property
\begin{eqnarray}
\label{curvature}\label{eq: bakry-emery}
    \Gamma P_t f & \leq & e^{-2\kappa t} P_t \Gamma f,
\end{eqnarray}
holds for all $t\ge 0$ and all functions $f\colon\dX\to\dR$. As a motivation, let us  recall what is perhaps the most emblematic application of the $\Gamma$-calculus developed by Bakry and \'Emery for the analysis of diffusions \cite{MR889476,MR1767995,MR3155209}: whenever the chain rule (\ref{eq: chain rule}) applies and $\kappa>0$, we have  
\begin{eqnarray}
\label{BELSI}
\tls & \leq & \frac{2}{\kappa}. 
\end{eqnarray}
This surprising relation between a geometric notion (the curvature) and a functional-analytic one (the log-Sobolev constant) is absolutely remarkable, and extending it beyond the context of Markov diffusions  has been an elusive goal. Here we use our main regularity estimate to bypass the lack of a chain rule and obtain the following discrete analogue of (\ref{BELSI}). 
\begin{theorem}[Curvature and LSI]\label{th:BE}When the Bakry-\'Emery curvature $\kappa$ is positive, we have
\begin{eqnarray*}
\label{approxBELSI}
\tls & \leq & \frac{33\log d}{\kappa}. 
\end{eqnarray*}
\end{theorem}
\begin{remark}[Sparse chains]For families of chains in which the sparsity parameter $d$ is uniformly bounded, Theorems \ref{th:LSMLS} and \ref{th:BE} give the same results as for diffusions, up to universal constants. This applies, in particular, to simple random walks on bounded-degree graphs.
\end{remark}
\begin{remark}[Sharpness]
As often, it is instructive to consider the rank-one case $T(x,y)=\pi(y)$, where $\pi$ is an arbitrary fully-supported probability measure. For this choice, we easily compute $\kappa=1/2$ and $d=\frac{1}{\pi_{\star}}$, where $\pi_\star$ denotes the smallest entry of $\pi$. On the other hand, it is classical that  $\tmls=\Theta(1)$ and $\tls=\Theta\left(\log \frac{1}{\pi_{\star}}\right)$, see \cite{MR1410112,MR2283379}. Thus, in this case, the reverse inequalities in Theorems \ref{th:LSMLS} and \ref{th:BE} actually also hold, albeit with different constants. 
\end{remark}
\begin{remark}[Reversibility]We have assumed throughout the paper that the transition matrix $T$ is reversible. Among other consequences, this ensures the validity of the identity (\ref{Dirichlet}), and the symmetry of the distance (\ref{def:dist}). However, some of our results have analogues in the non-reversible case. Indeed, we may replace $T$ by its additive reversibilization, defined by
\begin{eqnarray*}
\forall (x,y)\in\dX,\qquad \widehat{T}(x,y) & := & \frac{\pi(x)T(x,y)+\pi(y)T(y,x)}{2\pi(x)}.
\end{eqnarray*}
Note that the quantities $\E(\sqrt{f})$ and $\Ent(f)$ are not affected by this transformation, so that the log-Sobolev inequality (\ref{LSI}) remains the same. Thus, Theorem \ref{th: regularity of extremal} remains valid, albeit with $d$ being replaced with $\widehat{d}$. Also, the quantity $\cE(f,\log f)$ can not decrease by more than a factor $2$ upon replacing $T$ with $\widehat{T}$, so that $\widehat{\tmls}\le 2\tmls$. Consequently, Theorem \ref{th:LSMLS} remains valid in the non-reversible case, with  $\widehat{d}$ instead of $d$ and a twice larger absolute constant. 
\end{remark}

Over the past few years, there has been growing interest in developing alternative notions of curvature better suited to discrete state spaces; see, e.g., \cite{MR2484937,MR2989449,MR3701949,MR3906157,MR4253929,MR4507936} and the references therein. In particular, Ollivier \cite{MR2484937,MR2648269} proposed to replace the sub-commutation property (\ref{eq: bakry-emery}) with the condition
\begin{eqnarray}
\label{def:ricci}
\Lip(P_t f) & \le & e^{-\tilde{\kappa} t}\Lip(f),
\end{eqnarray}
which, by the classical Kantorovich-Rubinstein duality, amounts to an exponential contraction of the semi-group in the Wasserstein space $W_1(\dX)$. Since this property is often easier to verify in concrete models, it would be highly desirable to obtain a version of Theorem \ref{th:BE} featuring the Ollivier-Ricci curvature $\tilde{\kappa}$ instead of the Bakry-\'Emery curvature $\kappa$. 
\begin{conjecture}[Ollivier-Ricci curvature and LSI]\label{conjecture}There exists a universal constant $c<\infty$ such that whenever the Ollivier-Ricci curvature $\tilde{\kappa}$ is positive, we have 
\begin{eqnarray}
\tls & \leq & \frac{c\log d}{\tilde{\kappa}}. 
\end{eqnarray}
\end{conjecture}
In light of Theorem \ref{th:LSMLS}, the validity of this conjecture would follow from a celebrated prediction attributed to Peres and Tetali (see also \cite[Problem M]{MR2648269}), according to which
\begin{eqnarray*}
\label{PeresTetali}
\tmls & \leq & \frac{c}{\tilde{\kappa}},
\end{eqnarray*}
whenever $\tilde{\kappa}>0$, where $c$ is an absolute constant. Unfortunately, this prediction was recently disproved \cite{münch2023olliviercurvatureisoperimetryconcentration}, and Conjecture \ref{conjecture} seems to constitute a fairly reasonable replacement for it. Notably, the counter-example provided in \cite{münch2023olliviercurvatureisoperimetryconcentration} does not contradict our conjecture.  Interestingly, the Peres-Tetali conjecture was verified under an additional assumption known as non-negative sectional curvature \cite{münch2023olliviercurvatureisoperimetryconcentration,caputo2024entropycurvatureperestetaliconjecture}.  By virtue of Theorem \ref{th:LSMLS}, the validity of Conjecture \ref{conjecture} is guaranteed for such chains. However, non-negative sectional curvature is a rather restrictive assumption, whose removal  would be of substantial practical interest.

\section{Proofs}

\subsection{Proof of Theorem \ref{th: regularity of extremal}}
Let us start by equipping our state space $\dX$ with a natural metric, namely the graph distance induced by the set of allowed transitions:
\begin{eqnarray}
\label{def:dist}
\dist(x,y) & := & \min\left\{n\in\mathbb N\colon T^n(x,y)>0\right\}.
\end{eqnarray}
As usual, we let $\diam:=\max_{x,y\in\dX}\dist(x,y)$ denote the associated diameter. It turns out that the latter is directly controlled by the inverse log-Sobolev constant. 
\begin{lemma}[Diameter and LSI]\label{lem: bound diameter}
 We always have 
$
    \diam  \leq \sqrt{2}\tls.
$
\end{lemma}
Rather surprisingly, this simple observation seems to be new. Diameter bounds involving the log-Sobolev constant first appeared in the different context of compact manifolds with Ricci curvature bounded from below \cite{MR1292948,MR1307413,MR1346225}. Discrete versions that apply to our setup are discussed at length in Section 7.4 of the lecture notes \cite{MR1767995}. However, to the best of our knowledge,  all those bounds suffer from an additional dependency in one of the dimension parameters $|\dX|$, $d$, or $\pi_\star$. See, for example, \cite[Proposition 7.6]{MR1767995} and the estimate $\diam\le 64\tls\log d$ deduced from it, as well as its refinement $\diam\le 32\tmls\log d$ recently recorded in \cite{pedrotti2025newcutoffcriterionnonnegatively}. By virtue of Theorem \ref{th:LSMLS}, our estimate in Lemma \ref{lem: bound diameter} is always better, and the improvement can in fact be substantial: for example, in the emblematic case of simple random walk on the $n-$dimensional hypercube, the diameter estimate in Lemma \ref{lem: bound diameter} is sharp except for the $\sqrt{2}$ pre-factor, while the estimate $\diam\le 32\tmls\log d$ is off by a factor of order $\log n$. This example also demonstrates that our estimate is sometimes much better than the classical diameter bound using mixing times (see, e.g., \cite[Chapter 7]{MR3726904}).  

\begin{proof}[Proof of Lemma~\ref{lem: bound diameter}]
By virtue of a classical argument due to Herbst, the modified log-Sobolev inequality guarantees sub-Gaussian concentration under the stationary measure: for
any function $f\colon\dX \to \dR$ with $\EE[f]=0$ and $\Lip(f) \le 1$, and any $t \ge 0$, we have
\begin{eqnarray*}
\PP(f\ge t) & \le & e^{-\frac{t^2}{2\tmls}}.
\end{eqnarray*}
We refer the reader to the  lecture notes \cite[Section 2.3]{MR1767995} for a general presentation, and to \cite[Lemma 15]{MR4203344} for the precise discrete version used here. 
Choosing $t=\max f$ and noting that the left-hand side is at least $\pi_\star:=\min_{x}\pi(x)$, we obtain
\begin{eqnarray*}
(\max f)^2 & \le & 2\tmls\log\frac{1}{\pi_\star}\\
& \le & \frac{\tls^2}{2},
\end{eqnarray*}
where the second line uses (\ref{onesided}), as well as the bound $\log\frac{1}{\pi_\star}\leq \tls$ obtained by applying the log-Sobolev inequality (\ref{LSI}) to Dirac masses. Replacing $f$ with $-f$ yields the same estimate for $\min f$, and combining the two leads to
\begin{eqnarray*}
\max f-\min f & \le & \sqrt{2}\tls.
\end{eqnarray*}
At this point, our earlier requirement that $\pi(f)=0$ can simply be dropped, because the conclusion is invariant under shifting $f$ by a constant. In particular, we may take $f(x)=\dist(o,x)$ for any fixed state $o\in\dX$, and the claim readily follows.
\end{proof}
The second ingredient in the proof of Theorem \ref{th: regularity of extremal} is the following classical structural result,  stated and proved for instance  in the lecture notes \cite[Theorem 2.2.3]{MR1490046}, and which asserts that extremizers  in the log-Sobolev inequality (\ref{LSI}) must satisfy a specific functional equation, much like extremizers in the Poincaré inequality satisfy an eigenvalue equation. Both those results are, in fact, particular instances of Fermat's theorem, which asserts that the differential of a smooth function must vanish at interior extremizers. 
\begin{lemma}[Structure of log-Sobolev extremizers]\label{lm:extremizers}If a function $f\colon\dX\to(0,\infty)$ achieves equality in the log-Sobolev inequality (\ref{LSI}), then the function $g:=\sqrt{\frac{f}{\EE[f]}}$ must solve
\begin{eqnarray}
\label{local}
\tls \LL g + 2g\log g & = & 0.
\end{eqnarray}
Moreover, if there is no non-constant function achieving  equality in  (\ref{LSI}), then $\tls=2\trel=4\tmls$, where $\trel$ denotes the relaxation time (inverse spectral gap) of $\LL$.
\end{lemma}
We now have all we need to prove Theorem~\ref{th: regularity of extremal}.
\begin{proof}[Proof of Theorem~\ref{th: regularity of extremal}]Assume that $f\colon\dX\to(0,\infty)$ achieves equality in  (\ref{LSI}). By Lemma \ref{lm:extremizers}, the function $g:=\sqrt{\frac{f}{\EE[f]}}$ must then solve (\ref{local}). Now, consider a pair of states $(x,y)\in \dX^2$ realizing the definition of $\Lip(\log g)$. In other words, $T(x,y)>0$ and 
\begin{eqnarray*}
\log\left(\frac{g(y)}{g(x)}\right) & = & \Lip(\log g) \ := \ \ell.
\end{eqnarray*}
Recalling the definition of $d$ at (\ref{def:d}), we have in particular
\begin{eqnarray*}
\ell & \le & \log d+\log\left(\sum_{z\in\dX}T(x,z)\frac{g(z)}{g(x)}\right).
\end{eqnarray*}
On the other hand, evaluating (\ref{local}) at  $x$ and dividing through by $g(x)$, we know that
\begin{eqnarray*}
\sum_{z\in\dX}T(x,z)\frac{g(z)}{g(x)} & = & 1-\frac{2}{\tls}\log g(x).
\end{eqnarray*}
But $\EE[g^2]=1$ by construction, so $\max g\geq 1$, and we may thus write
\begin{eqnarray*}
-\log g(x) & \le & \max(\log g)-\min(\log g) \\
& \le & \ell\diam \\
& \le & \sqrt{2}\ell\tls,
\end{eqnarray*}
where the last line uses Lemma \ref{lem: bound diameter}. Combining the last three displays, we arrive at
\begin{eqnarray*}
\ell & \le & \log d+\log \left(1+2\sqrt{2}\ell\right)\\
& \le & \log d +\log\left(1+\frac{\ell}2\right)+\frac{5}{2}\log 2\\
& \le & \frac{\ell+7\log d}2,
\end{eqnarray*}
where the last line uses $\log (1+u)\le u$ and the fact that $d\ge 2$. Thus, we obtain
\begin{eqnarray*}
\ell & \le & 7\log d,
\end{eqnarray*}
which concludes the proof since $\Lip(\log f)=2\Lip(\log g)=2\ell$. 
\end{proof}
\subsection{Proof of  Theorem \ref{th:LSMLS}}
Let us now turn to the proof of Theorem \ref{th:LSMLS}. Consider a function $f\colon\dX\to(0,\infty)$ which achieves equality in the log-Sobolev inequality (\ref{LSI}). Then, 
\begin{eqnarray*}
\tls\E(\sqrt{f}) & = & \Ent(f)\\
& \le & \tmls \E(f,\log f)\\
& \le & \tmls\phi(\Lip(\log f))\E(\sqrt{f}),
\end{eqnarray*}
where we have successively used the modified log-Sobolev inequality (\ref{MLSI}) and the approximate chain rule (\ref{approxlogchainrule}). If $f$ is not constant, we may safely simplify through by $\E(\sqrt{f})$ to obtain 
\begin{eqnarray*}
\tls  & \le & \tmls\phi(\Lip(\log f)).
\end{eqnarray*}
This is enough to conclude, since by Theorem \ref{th: regularity of extremal} and the fact that $\phi$ is increasing, we have
\begin{eqnarray*}
\phi(\Lip(\log f))&  \le & \phi(14\log d)\\
& \le & 15\log d,
\end{eqnarray*}
where in the second line we have observed that $\frac{\phi(r)}{r}=\frac{e^{r/2+1}}{e^{r/2}-1}$ is a decreasing function of $r$, which equals $\frac{129}{127}\le\frac{15}{14}$ when $r=14\log 2$. Finally, in the degenerate case where there is no non-constant function  $f\colon\dX\to(0,\infty)$ achieving equality in the log-Sobolev inequality (\ref{LSI}), the second part of Lemma \ref{lm:extremizers} ensures that $\tls=4\tmls$, which is more than enough to conclude. 

\subsection{Proof of Theorem \ref{th:BE}}
We will need three lemmas. The first one is a  uniform control on the log-Lipschitz constant of $P_t f$ in terms of that of $f$, for any positive observable $f$ and any time $t\ge 0$. 
\begin{lemma}[Regularity along the semi-group]\label{lm:regularity} For any $f\colon\dX\to(0,\infty)$ and $t\ge 0$, we have
\begin{eqnarray*}
\Lip(\log P_tf) &  \le & \Lip(\log f)+\log d.
\end{eqnarray*}
\end{lemma}
\begin{proof}
Fix $f\colon\dX\to(0,\infty)$, and write $\ell:=\Lip(\log f)$. For any $x\in\dX$ and any $n\ge 1$, we have
\begin{eqnarray*}
T^{n}f(x) &  = & \sum_{z,w\in\dX}T^{n-1}(x,z)T(z,w)f(w)\\
& \le & \sum_{z,w\in\dX}T^{n-1}(x,z)T(z,w)f(z)e^{\ell}\ = \ e^{\ell} T^{n-1}f(x),
\end{eqnarray*}
by definition of $\ell$. On the other hand, for any neighbor $y$ of $x$, we also have
\begin{eqnarray*}
T^{n}f(y) &  = & \sum_{z\in\dX}T(y,z)T^{n-1}f(z)\\
& \ge & T(y,x)T^{n-1}f(x)\\
& \ge & d^{-1}T^{n-1}f(x).
\end{eqnarray*}
Combining those two bounds, we obtain 
\begin{eqnarray*}
T^{n}f(x) & \le & de^{\ell}T^{n}f(y).
\end{eqnarray*}
This was established for $n\ge 1$, but the conclusion actually also holds when $n=0$, by definition of $\ell$. Finally, averaging this inequality over the variable $n\in\mathbb N$ according to the Poisson distribution with parameter $t\ge 0$ yields
\begin{eqnarray*}
P_tf(x) & \le & de^{\ell}P_tf(y).
\end{eqnarray*}
Since this holds for any neighbors $x,y\in\dX$, the claim is proved. 
\end{proof}
Our second lemma is a general bound on the Dirichlet form, which only uses reversibility.
\begin{lemma}[Upper-bound on the Dirichlet form]\label{lm:dirichlet}For any $f,g\colon\dX\to\dR$, we have
\begin{eqnarray*}
   \E(f^2,g ) & \le & 2\sqrt{\E(f){\EE\left[f^2\Gamma g\right]}}.
\end{eqnarray*}
\end{lemma}
\begin{proof}
Using reversibility, we may write
\begin{eqnarray*}
   \E(f^2,g )& = & \frac12 \sum_{x,y\in \dX} \pi(x)T(x,y)(f^2(x)-f^2(y))(g(x)-g(y)) \\
   & = &\frac12 \sum_{x,y\in \X} \pi(x)T(x,y)(f(x)+f(y))(f(x)-f(y))(g(x)-g(y)) \\
   & = & \sum_{x,y\in \X} \pi(x)T(x,y)f(x)(f(x)-f(y))(g(x)-g(y))\\
   & = & 2\EE \left[f\Gamma(f,g)\right].
\end{eqnarray*}
The claim now readily follows from the two bounds
\begin{eqnarray*}
\Gamma(f,g) & \le & \sqrt{\Gamma f\Gamma g}\\
\EE[f\sqrt{\Gamma f \Gamma g}] & \le & \sqrt{\EE[f^2\Gamma g]\EE[\Gamma f]},
\end{eqnarray*}
which are two particular instances of the Cauchy-Schartz inequality.
\end{proof}
Finally, our third ingredient is a result borrowed from \cite[Lemma~2]{pedrotti2025newcutoffcriterionnonnegatively}, which constitutes an approximate discrete version of the chain rule $\EE[f\Gamma \log f]=\E(f,\log f)$ valid for diffusions. 
\begin{lemma}[Approximate chain rule]\label{lm:approxchain}For any $f\colon\dX\to(0,\infty)$, we have
\begin{eqnarray*}
\EE\left[f\Gamma \log f\right] & \le & \left(1+\Lip(\log f)\right)\cE(f,\log f).
\end{eqnarray*}
\end{lemma}
\begin{proof}[Proof of Theorem~\ref{eq: bakry-emery}]
Fix $f\colon\dX\to(0,\infty)$ and $t\ge 0$. By reversibility and Lemma \ref{lm:dirichlet}, we have
\begin{eqnarray*}
\E( P_t f,\log P_tf) & = & \E(f,P_t\log P_tf)\\
& \le & 2\sqrt{\E(\sqrt{f})\EE\left[f\Gamma P_t\log P_t f\right]}.\end{eqnarray*}
On the other hand, 
\begin{eqnarray*}
\EE\left[f\Gamma P_t\log P_t f\right] & \le & e^{-2\kappa t}\EE\left[fP_t\Gamma\log P_t f\right]\\
& = & e^{-2\kappa t}\EE\left[(P_tf)\Gamma \log P_t f\right]\\
& \le & e^{-2\kappa t}\left(1+\Lip(\log P_tf)\right)\E(P_tf,\log P_tf)\\
& \le & e^{-2\kappa t}\left(1+\log d+\Lip(\log f)\right)\E(P_tf,\log P_tf),
\end{eqnarray*}
where we have successively used the Bakry-\'Emery curvature condition (\ref{eq: bakry-emery}), reversibility, Lemma \ref{lm:approxchain}, and Lemma \ref{lm:regularity}. Combining those two observations, we deduce that
\begin{eqnarray*}
\E( P_t f,\log P_tf) & \le & 4e^{-2\kappa t}\left(1+\log d+\Lip(\log f)\right)\E(\sqrt{f}).
\end{eqnarray*}
Finally, since this is valid for all $t\in[0,\infty)$, we may integrate with respect to $t$ to obtain
\begin{eqnarray*}
       \Ent(f) & \le & \frac{2}{\kappa}\left(1+\log d+\Lip(\log f)\right)\E(\sqrt{f}).
\end{eqnarray*}
Since this holds for any function $f\colon\dX\to(0,\infty)$, we may finally choose one that achieves equality in the log-Sobolev inequality (\ref{LSI}). In view of Theorem \ref{th: regularity of extremal}, we obtain
\begin{eqnarray*}
       \tls\cE(\sqrt{f}) & \le & \frac{2}{\kappa}(1+15\log d)\E(\sqrt{f}),
\end{eqnarray*}
which yields the desired conclusion provided $f$ is not constant. On the other hand, if there is no non-constant function achieving equality in (\ref{LSI}), then Lemma \ref{lm:extremizers} guarantees that $\tls=2\trel$, which in view of the classical bound $\trel\le \frac{1}{\kappa}$, is more than enough to conclude.
\end{proof}

\section*{Acknowledgment.}This work was initiated during a visit of P.Y.  at Université Paris-Dauphine as an Invited Professor. J.S. acknowledges support from the ERC consolidator grant CUTOFF (101123174). Views and opinions expressed are however those of the authors only and do not necessarily reflect those of the European Union or the European Research Council Executive Agency. Neither the European Union nor the granting authority can be held responsible for them.

\bibliographystyle{plain}
\bibliography{draft}
\section*{Author affiliations}
Justin Salez: CEREMADE, Université Paris-Dauphine \& PSL, Place du Maréchal de Lattre de Tassigny, F-75775 Paris Cedex 16, FRANCE. 
\texttt{\small e-mail:  justin.salez@dauphine.psl.eu}\\

\noindent Pierre Youssef: 
Division of Science, NYU Abu Dhabi, Saadiyat Island, Abu Dhabi, UAE \& Courant
Institute of Mathematical Sciences, New York University, 251 Mercer st, New York,
NY 10012, USA. 
\texttt{\small e-mail:  yp27@nyu.edu}
\end{document}